\documentclass[12pt]{article}
\usepackage{graphicx}
\usepackage{amssymb}
\usepackage{amsmath}
\usepackage{amsthm}

\newtheorem{lemma}{Lemma}
\newtheorem{proposition}{Proposition}

\newtheorem{theorem}{Theorem}

\newcommand{\Sph}{{\cal S}}
\newcommand{\interior}{\mathop{\mathrm{int}}}
\newcommand{\aff}{\mathop{\mathrm{aff}}}
\newcommand{\co}{\mathop{\mathrm{co}}}
\newcommand{\rec}{\mathop{\mathrm{rec}}}
\newcommand{\lspan}{\mathop{\mathrm{span}}}
\newcommand{\cone}{\mathop{\mathrm{cone}}}

\newcommand{\R}{\mathbb{R}}

\title{Facially exposed cones are not always nice}
\author{Vera Roshchina\thanks{CRN, University of Ballarat, Australia; e-mail: vroshchina@ballarat.edu.au.}}

\begin{document}

\maketitle

\begin{abstract} We address the conjecture proposed by G\'abor Pataki that every facially exposed cone is nice. We show that the conjecture is true in the three-dimensional case, however, there exists a four-dimensional counterexample of a cone that is facially exposed but is not nice.
\end{abstract}

\section{Introduction}

A closed convex cone $K\subset \R^n$ is called {\em nice} if for every face $F$ of $K$ the Minkowski sum $K^*+F^\perp$ is closed (here $F^\perp$ is the orthogonal complement to the linear span of $F$, and $K^*$ is the dual cone of $K$). Such cones are also called facially dual complete~\cite{TW2012}. Nice cones provide a simple characterisation of the closedness of the linear image of the dual of a closed convex cone \cite{PatakiClLinIm}, they feature in the study of lifting of convex sets~\cite{Lifts},  and in the analysis of the facial reduction algorithm \cite{BorWolk,PatakiFacRed}. Facial exposedness is another classical notion, and it is important in the analysis of optimization problems (e.g. see \cite{TamSchneider,WolkEtAlSDP}). It is known that homogeneous \cite{TruongTuncel} and projectionally exposed cones \cite{SungTamProj} are facially exposed, and that niceness is preserved under SDP representations \cite{CT2008}. The standard cones used in optimization (nonnegative orthant, second-order cone and the cone of positively semidefinite matrices) are both facially exposed and nice. 

In \cite{PatakiFExpNice} a systematic study of nice cones is undertaken. In particular, it is shown that every nice cone is facially exposed, and that niceness is guaranteed by facial exposedness with an additional condition that for all $F$ faces of $K$ and all $H$ faces of $F^*$ which are minimal and distinct from $F^\perp$, the face  $H$  is exposed. It is shown that this condition is not necessary for niceness though, which leads to the conjecture that all facially exposed cones are nice. The goal of this work is to demonstrate that while the conjecture is true in up to three dimensions, it fails in general: we show that there exists a four-dimensional convex closed cone that is facially exposed but is not nice. 

Our paper is organised as follows: in Section~\ref{sec:prelim} we remind some essential definitions and prove that all three-dimensional facially exposed cones are nice. In Section~\ref{sec:tech} we obtain some general results, which we use in Section~\ref{sec:4D} to analyse the four-dimensional counterexample. Throughout the paper, by $\R^n$ we denote the $n$-dimensional Euclidean space; for $x,y\in \R^n$, we let $\langle x,y\rangle := \sum_{i=1}^n x_i y_i$, and $\|x\|: = \sqrt{\langle x,x \rangle}$ denote the scalar product and Euclidean norm. By $\Sph_{n-1}$ we denote the unit sphere in the relevant $n$-dimensional space. For a set $C\subset \R^n$ by $\aff C$, $\co C$ and $\cone C$ we denote, respectively, its affine, convex and conic hulls.

\section{Preliminaries and the three-dimensional case}\label{sec:prelim}

A nonempty convex subset $F$ of a convex closed set $C\subset \R^n$ is called a {\em face} of $C$ if $\alpha x +(1-\alpha) y\in F$ with $x,y\in C$ and $\alpha \in (0,1)$ imply $x,y\in F$. 
We use the standard notation $F\unlhd C$ to denote that $F$ is a face of $C$. Observe that $C$ is its own face. We say that $F\unlhd C$ is {\em proper} if $F\neq C$. 

When $K\subset \R^n$ is a closed convex cone, a face can be defined equivalently as a  convex closed subset $F$ of $K$ such that $x+y\in F$ with $x,y\in K$ imply $x,y\in F$.

A face $F$  of a closed convex set  $C\subset \R^n$ is called {\em exposed} if it can be represented as the intersection of $C$ with a supporting hyperplane, i.e. there exist $y\in \R^n$ and $d\in \R$ such that for all $x\in C$
\begin{equation}\label{eq:defExposed}
\langle y,x\rangle \leq  d \quad \forall \, x\in C;\qquad  \langle y,x\rangle = d\quad  \text{iff}\quad  x\in F.
\end{equation}
We say that a pair $(y,d)\in \R^{n+1}$ exposes $F\unlhd C$ if $(y,d)$ satisfy  \eqref{eq:defExposed}. A set is {\em facially exposed} if all of its faces are exposed. 

Observe that for every hyperplane $H=\{x\in \R^n\,|\, \langle x,y\rangle = d\}$ supporting a closed convex set $C\subset \R^n$ the set $C\cap H$ is a face of $C$; moreover, $C$ is always an exposed face of itself (letting $(y,d) = (0_n,0)$). It is not difficult to observe that for a cone $K$ and any pair $(y,d)$ exposing a face $F\unlhd K$ we have $d=0$.

Let $C$ be a convex set in $\R^n$. The {\em polar} of $C$ is the set 
$$
C^\circ = \{y\in \R^n\, |\, \sup_{x\in C}\langle x,y\rangle \leq 1\}.
$$
When $K\subset \R^n$ is a cone, the polar set coincides with the {\em polar cone} of $K$:
$$
K^\circ = \{y\in \R^n\, |\, \sup_{x\in K}\langle x,y\rangle \leq 0\}.
$$ 
The {\em dual cone} $K^*$ of $K$ is $K^* = -K^\circ$. Observe that $K^\circ$ and $K^*$ are always closed.

A cone $K$ is called {\em nice} if for every face $F$ of $K$ the set $K^*+F^\perp$ is closed, where $F^\perp$ is the orthogonal complement to $\lspan F$. 

We will also use the notion of {\em recession cone} 
$$
\rec C : = \{d\in \R^n\,|\, \exists x\in C \quad x+td \in C \quad \forall \, t\geq 0\}.
$$

To prove that every facially exposed set in the three-dimensional space is nice, we will need the following two trivial statements

\begin{lemma}\label{lem:SumSubspace} Let $C\subset \R^n$, and let $L$ be a linear subspace of $\R^n$. Then
$$
C+L^\perp = \Pi_{L} C +L^\perp,
$$ 
where by $\Pi_L$ we denote the orthogonal projection onto the linear subspace $L$.
\end{lemma}

\begin{lemma}\label{lem:sumClosedSetSubspace} Let $L\subset \R^n$ be a linear subspace, and assume $C\subset L$ is closed. Then the set $C+L^\perp$ is closed.
\end{lemma}

\begin{theorem}\label{thm:3Donly} A convex closed facially exposed cone $K\subset \R^3$ is nice.
\end{theorem}
\begin{proof} Let $K\subset \R^3$ be a facially exposed convex closed cone, and let $F \unlhd K$. Observe that if $\lspan F = \R^3$, we have $F^\perp = \{0\}$, and the set $K^*+F^\perp = K^*$ is closed since the dual cone is closed. Likewise, if $\lspan F = \{0\}$, then $F^\perp = \R^3$ and $K^*+F^\perp = \R^3$ is closed. In the case when $F$ is one-dimensional, i.e. $\lspan F = \lspan \{l\}$ for some $l\neq 0$, observe that $\Pi_{\lspan\{ l\}} K^*$ is a one-dimensional cone that contains zero, which is always closed. Hence, using Lemmas~\ref{lem:SumSubspace} and \ref{lem:sumClosedSetSubspace} the set
$$
K^*+F^\perp = \Pi_{\lspan \{l\}}+F^\perp
$$
is closed. It remains to consider the case when $F$ is two-dimensional, i.e. $F = \cone \{p_1, p_2\}$, where $p_1, p_2 \in \Sph_2$ are non-collinear. Observe that $E_i := \cone \{p_i\}$ is a face of $K$ for each $i\in \{1,2\}$ (see Fig.~\ref{fig:03}). 
\begin{figure}[ht]
\centering
\includegraphics[scale=1]{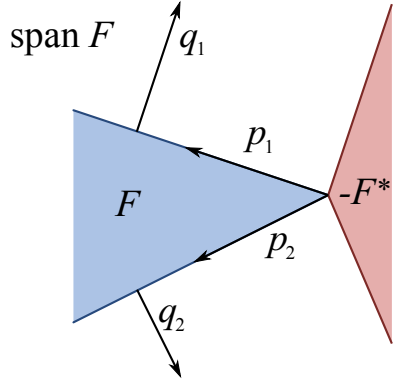}
\caption{Illustration to the proof of Theorem~\ref{thm:3Donly}.}
\label{fig:03}
\end{figure}
Since $K$ is facially exposed, there exist $h_1, h_2\in K^*$ such that 
$(-h_i,0)$ exposes $E_i$, $i\in \{1,2\}$. Observe that $h_i\notin F^\perp$ (otherwise $(h_i,0)$ would expose the whole face $F$). Let $q_i: = \Pi_{\lspan F} h_i$, and observe that $q_i \neq 0$ since $h_i\notin F^\perp$. We have 
$$
\langle q_i, p_i\rangle=\langle h_i, p_i\rangle = 0 \quad \forall i\in \{1,2\};\qquad \langle q_i, p_j\rangle =\langle h_i,p_j\rangle >0 \quad \forall i\neq j. 
$$
This yields (see Fig.~\ref{fig:03})
\begin{equation}\label{eq:006}
F^* =  \cone\{q_1,q_2\}+F^\perp.
\end{equation} 
Observing that $\Pi_{\lspan F} \cone \{h_1,h_2\} = \cone\{ q_1, q_2\}$, from Lemma~\ref{lem:SumSubspace} we have 
\begin{equation}\label{eq:007}
\cone\{q_1,q_2\}+F^\perp = \cone\{h_1,h_2\}+F^\perp.
\end{equation}
Finally, \eqref{eq:006} and \eqref{eq:007} yield $F^* = \cone\{h_1,h_2\}+F^\perp \subset K^*+F^\perp $. Since $F^* = \overline{(K^*+F^\perp)}$ (see \cite[Remark~1]{PatakiFExpNice}) this yields  $K^*+F^\perp = F^*$, hence, $K^*+F^\perp$ is closed. By the arbitrariness of $F$ it follows that the cone $K$ is nice.
\end{proof}

\section{Technical results}\label{sec:tech}

We next prove two fairly trivial results that establish relations between the faces of a closed convex set $C\subset \R^n$ and its homogenization $K =\cone( \{1\}\times C)\subset \R^{n+1}$. 

\begin{proposition}\label{prop:FacesKfromC} Let $C\subset \R^{n}$ be a nonempty compact convex set, and let 
\begin{equation}\label{eq:025}
K = \cone(\{1\}\times C)\subset \R^{n+1}.
\end{equation}
Then $K$ is a closed convex cone and the only faces of $K$ are $\{0_{n+1}\}$ and the sets
$$
F_K = \cone \{\{1\}\times F_C\},\quad F_C\unlhd C.
$$ 
\end{proposition}
\begin{proof} First of all, observe that \eqref{eq:025} yields the following two relations that we use throughout the proof:
$$
K = \bigcup_{\lambda \geq 0, x\in C} \lambda(1,x); \qquad C= \{x\,|\, (1,x)\in K\}.
$$
Observe that $K$ is a convex cone by definition. From \cite[Lemma~5.41]{Guler} we have $\overline{K}  = \{\lambda (1,x)\, |\, x\in C,\lambda>0\}\cup \{(0,d)\,|\, d\in \rec C\} $. Since $C$ is compact, its recession cone is trivial, hence, $\overline{K} = K+\{0\} = K$, i.e. $K$ is closed.  

For every $x=(x_0,\bar x)\in K$ (where $x_0\in \R$, $\bar x\in \R^n$) we have 
$x_0 \geq 0$, and $x_0 = 0$ yields $x=0_{n+1}$. We will use this observation several times in the proof.

Let $F_C \unlhd C$ and $ F_K := \cone \{\{1\}\times F_C\}$.
We show that $F_K \unlhd K$. Pick a $z=(z_0,\bar z)\in F_K$, and let $x=(x_0,\bar x),y=(y_0, \bar y)\in K$ be such that $z=x+y$. If $x_0=0$, we have $x=0$, hence, $z=y\in F_K$. Similarly, $y_0=0$ yields $z=x\in F_K$. Assume that both $x_0$ and $y_0$ are not zero. Then $x_0,y_0,z_0>0$. 
Let 
$$
\bar z' = \frac{1}{z_0}\bar z, \quad \bar x' = \frac{1}{x_0}\bar x, \quad \bar y'= \frac{1}{y_0}\bar y.
$$
Observe that $\bar x',\bar y'\in C$ by the definition of $K$; moreover, $\bar z'\in F_C$ and
$$
\bar z' = \frac{x_0}{z_0} \bar x' +\frac{y_0}{z_0} \bar y', \quad \frac{x_0}{z_0}+\frac{y_0}{z_0}=1,\quad  \frac{x_0}{z_0},\frac{y_0}{z_0}>0,
$$
hence, $\bar z'$ is a convex combination of $\bar x'$ and $\bar y'$. Since $\bar z'\in F_C$, and $F_C\unlhd C$, this yields $\bar x',\bar y'\in F_C$, hence, $x,y\in F_K$, and by the arbitrariness of $x,y,z$ 
the set $F_K$ is a face of $K$. It is not difficult to observe that $\{0_{n+1}\}$ is a face of $K$ as $\{x=(x_0,\bar x)\,|\, x_0=0\}\cap K = \{0_{n+1}\}$.

Now assume $F_K$ is a face of $K$. Let 
$$
F_C := \{x\,|\, (1,x )\in F_K \}.
$$
In the case when $F_C = \emptyset$, for every $x\in F_K$ we have $x_0 = 0$, hence, $x=0$, and therefore $F_K = \{0_{n+1}\}$. Consider the case when $F_C\neq \emptyset$. Then $F_K = \cone \{\{1\}\times F_C\}$, and it remains to show that $F_C$ is a face of $C$. Pick any point $\bar z'\in F_C$, and let $\bar z'=\alpha \bar x'+(1-\alpha)\bar y'$, where $\bar x',\bar y'\in C$ and $\alpha \in (0,1)$. Since $\bar z'\in F_C$, the point $z = (1, \bar z')\in F_K$. Let $x = \alpha (1,\bar x')$, $y = (1-\alpha) (1, \bar y')$. Observe that $z = x+y$. Since $F_K$ is a face, 
$x,y\in F_K$, therefore, $\bar x',\bar y'\in F_C$, and hence $F_C$ is a face of $C$.
\end{proof}

\begin{proposition}\label{prop:IfCFExpKFExp} Let $C$ be a compact convex set in $\R^n$. If $C$ is facially exposed, then so is $K = \cone \{\{1\}\times C\}\subset \R^{n+1}$.
\end{proposition}
\begin{proof} 
First observe that $\{0_{n+1}\}$ is an exposed face of $K$, since $\{x=(x_0,\bar x)\,|\, x_0=0\}\cap K = \{0_{n+1}\}$. By  Proposition~\ref{prop:FacesKfromC} the only remaining faces of $K$ are 
$$
F_K = \cone \{\{1\}\times F_C\}, \quad F_C \unlhd C.
$$
Assume that $F_C \unlhd C $. Since $F_C$ is exposed, there exist $\bar y\in \R^{n}$ and $d\in \R$ such that 
\begin{equation}\label{eq:008}
\langle \bar x, \bar y \rangle < d \quad \forall\,\bar x\in C\setminus F_C; 
\qquad 
\langle \bar x, \bar y \rangle = d \quad \forall\,\bar x\in F_C. 
\end{equation}
Let $y := (-d,\bar y) \in \R^{n+1}$. Pick any $ x=(x_0,\bar x)\in F_K$. If $x_0 = 0$, we have $\langle x,y \rangle = 0 $. If $x_0>0$, observe that 
$\bar x' = \tfrac{1}{x_0}\bar x \in F_C$, hence, we have
$$
\langle x, y\rangle = x_0\left(-d+\langle \bar y,\bar x'\rangle\right)= x_0 (-d+d)= 0. 
$$
Now suppose $x\in K\setminus F_K$. Then $x_0\neq 0$ and  $\bar x':= \frac{1}{x_0}\bar x\notin F_C$. We have from \eqref{eq:008}
$$
\langle x, y\rangle = x_0\left(-d+\langle \bar y,\bar x'\rangle\right)< x_0 (-d+d)= 0. 
$$
We have therefore shown that the pair $((-d,y),0)$ exposes $F_K \unlhd K$.
\end{proof}

The following statement relates the polar sets of $C\subset \R^n$ and $K=\cone(\{1\}\times C)\subset \R^{n+1}$.

\begin{proposition}\label{prop:PolarCone} Let $C\subset \R^n$ be a compact convex set such that $0_n\in \interior C$, and let $K = \cone(\{1\}\times C)$. Then
$$
K^\circ = \cone \{\{-1\}\times C^\circ\}.
$$
\end{proposition} 
\begin{proof}
From the definition of a polar cone 
\begin{align*}
K^\circ & = \{(y_0, \bar y)\,|\, \sup_{\alpha\geq 0,\, \bar x\in C}\alpha (y_0+\langle \bar x, \bar y\rangle)\leq 0\}\\
 & = \{(y_0, \bar y)\,|\, \sup_{\alpha> 0,\, \bar x\in C}\alpha (y_0+ \langle \bar x, \bar y\rangle)\leq 0\}\\
 & = \{(y_0, \bar y)\,|\, \sup_{\bar x\in C}(y_0+ \langle \bar x, \bar y\rangle)\leq 0\}\\
 & = \{(y_0, \bar y)\,|\, \sup_{\bar x\in C}\langle \bar x, \bar y\rangle\leq -y_0\}.
\end{align*}
Observe that since $0_n \in \interior C$, for every $y\in \R^n $  we have $\sup_{x\in C}\langle x, y\rangle >0$, hence, 
\begin{align*}
K^\circ & = \{\alpha(-1, \bar y), \alpha \geq  0\,|\, \alpha\sup_{\bar x\in C}\langle  \bar x, \bar y\rangle\leq \alpha\}\\
& = \{0_{n+1}\}\cup \{\alpha(-1, \bar y), \alpha >  0\,|\, \sup_{\bar x\in C}\langle  \bar x, \bar y\rangle\leq 1\}\\
& = \{0_{n+1}\}\cup \{\alpha(-1, \bar y), \alpha >  0\,|\, \bar y\in C^{\circ}\}\\
& = \cone \{\{-1\}\times C^\circ\}.
\end{align*}
\end{proof}

\section{Four-dimensional Counterexample}\label{sec:4D}

The goal of this section is to prove the following result

\begin{theorem}\label{thm:counterex} There exists a facially exposed closed convex cone $K\subset \R^4$ such that $K$ is not nice 
\end{theorem}

We first describe the construction of our counterexample. Let the three-dimensional curves $\gamma_i:[0,T]\to \R^3$, with $T = \pi/4$, $i\in I := \{1,2,3,4\}$ be defined as follows:
\begin{align}\label{eq:curves}
\gamma_1 (t) & := \left(0, - \sin t,  \cos t -1\right); 
&
\gamma_2 (t) & := \left(0, \cos t-1, - \sin t \right);\notag\\
\gamma_3 (t) & := \left(-\sin t, 1- \cos t,  0 \right); 
& 
\gamma_4 (t) & := \left(\cos t -1, \sin t, 0\right).
\end{align}
For convenience, we use $\gamma_i$ to denote the whole segment $\gamma_i([0,T])$. Let 
$ C: = \co_{i\in I} \{\gamma_i\}$.

The three-dimensional set $C$ is shown in Fig.~\ref{fig:00}, and in Fig.~\ref{fig:02} we give some geometric details related to the construction of $C$.
\begin{figure}[ht]
\centering
\includegraphics[scale=0.7]{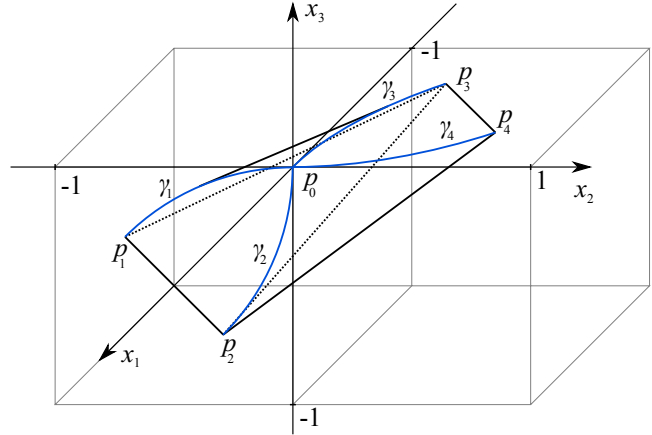}
\caption{The three-dimensional set $C = \co \{\gamma_1, \gamma_2, \gamma_3, \gamma_4\}$.}
\label{fig:00}
\end{figure}
\begin{figure}[ht]
\centering
\includegraphics[scale=.4]{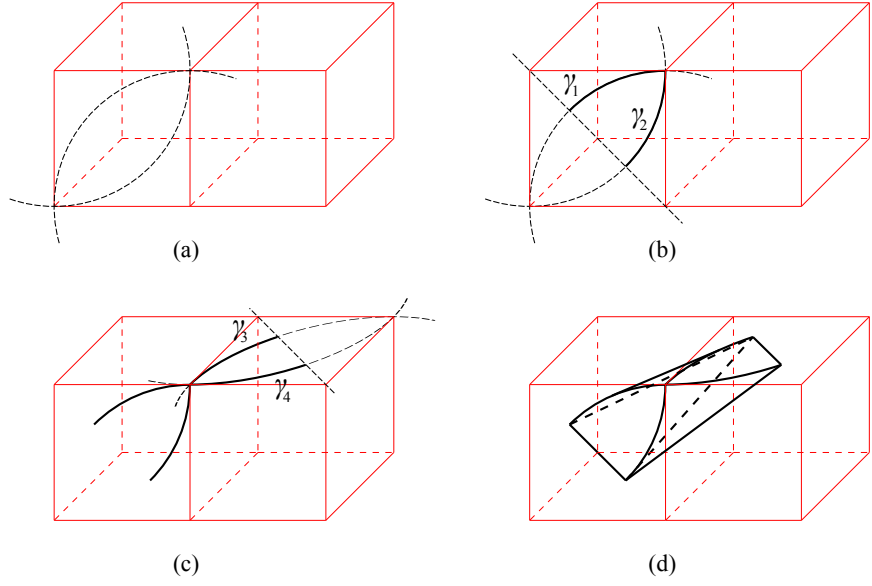}
\caption{The geometric construction of the set $C$: (a) consider two adjacent cubes as shown; 
on the frontal surface of the left cube draw two circles with centres at the cube's top-left and bottom-right corners with the radius equal to the 
edge; (b) intersect the circles with the diagonal through their centres; $\gamma_1$ and $\gamma_2$ are the segments of the circles bounded by their intersections with the diagonal and the top-right corner of the face; (c) repeat the same process on the top face of the right cube to obtain the curves $\gamma_3$ and $\gamma_4$; (d) construct the convex hull.}
\label{fig:02}
\end{figure}
To finish the construction of the counterexample, let 
\begin{equation}\label{eq:DefK}
C' : = 2C+c   \quad \text{ with } \quad c = \left(\tfrac{1}{2},0,\tfrac{1}{2}\right)\quad \text{and} \quad K  = \cone (\{1\}\times C'\}.
\end{equation}
The set $C'$ is shown in Fig.~\ref{fig:01}.
\begin{figure}[ht]
\centering
\includegraphics[scale=0.7]{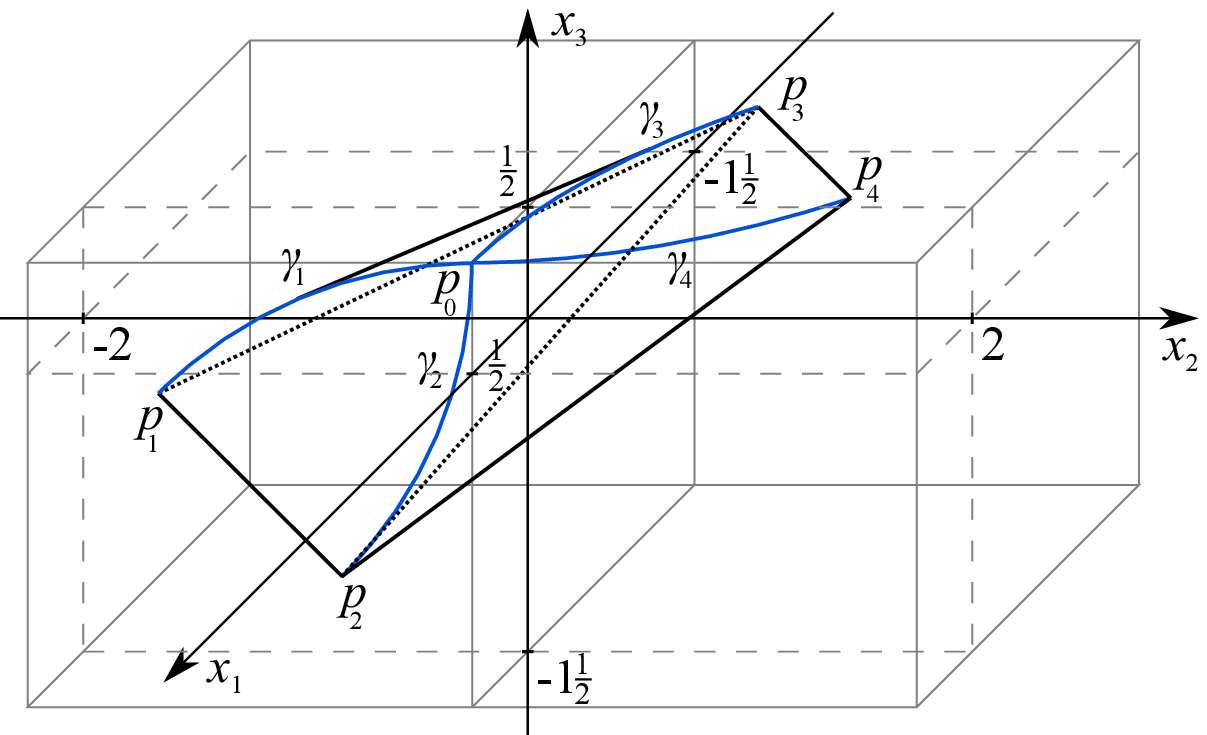}
\caption{The three-dimensional set $C' = 2 C+\{c\}$.}
\label{fig:01}
\end{figure}
We prove Theorem~\ref{thm:counterex} by showing that this set is a convex closed facially exposed cone which is not nice. Throughout this section, we always use the notation $C$, $C'$ and $K$ to refer to the aforementioned sets. 

For the proof of Theorem~\ref{thm:counterex} we need several technical statements related to the geometry of the sets $C$, $C'$ and $K$. 

For convenience, denote the endpoints of the curves in \eqref{eq:curves} as follows.
\begin{eqnarray}\label{eq:defp}
 &  p_0:  =0_3  = \gamma_1(0) = \gamma_2(0)= \gamma_3(0) = \gamma_4(0);\notag\\
 &  p_1:  =\left(0, -\tfrac{1}{\sqrt{2}}, \tfrac{1}{\sqrt{2}}-1\right)  = \gamma_1(T);\qquad 
 p_2:  =\left(0, \tfrac{1}{\sqrt{2}}-1, -\tfrac{1}{\sqrt{2}}\right)  = \gamma_2(T);\\
 &  p_3:  =\left(-\tfrac{1}{\sqrt{2}}, 1-\tfrac{1}{\sqrt{2}},0\right)  = \gamma_3(T);\qquad
 p_4:  =\left(\tfrac{1}{\sqrt{2}}-1,\tfrac{1}{\sqrt{2}},0\right)  = \gamma_4(T).\notag
\end{eqnarray}
Throughout this section, we also utilize the following notation.
\begin{align}\label{eq:DefThetaThingies}
t_\theta &:=  \arccos\left(\frac{\sin \theta}{1+\sin \theta - \cos \theta}\right);\notag\\
y_\theta &:=  (-\sin t_\theta \sin \theta, - \cos t_\theta \sin \theta, \cos t_\theta \cos \theta);\\
d_\theta & : = \cos t_\theta(1-\cos \theta)=\sin \theta (1-\cos t_\theta).
\end{align}

\begin{proposition}\label{prop:TTheta} Let 
$$
\varphi (\theta) : = \frac{\sin \theta}{1+\sin \theta - \cos \theta}, \quad \theta \in (0,T].
$$
The function $\varphi: (0,T]\to \R$ is strictly decreasing; moreover, 
$$
\lim_{\theta \downarrow 0} \varphi (\theta)  = 1; \qquad \varphi(T) = \frac{1}{\sqrt{2}}.
$$
Hence, the mapping $t_\theta :(0,T]\to (0,T]$, $t_\theta  = \arccos \varphi(\theta)$ is a bijection.
\end{proposition}
\begin{proof}
Observe that 
$$
\varphi'(\theta) = \frac{\cos \theta -1}{(1+\sin \theta -\cos \theta)^2}<0 \quad \forall \, \theta \in (0,T],
$$
hence, $\varphi$ is strictly decreasing on $(0, T]$; using l'H\^opital's rule, we have
\begin{equation}\label{eq:012}
\lim_{\theta\downarrow 0}\varphi(\theta) = \frac{\cos \theta}{\cos\theta+\sin \theta} = 1 =\cos (0); \quad \varphi(T) = \frac{1}{\sqrt{2}}= \cos T.
\end{equation}
It is evident from the strict monotonicity of $\varphi$ and \eqref{eq:012} that $t_\theta$ is bijective.
\end{proof}

In the next statement we list all proper faces of the set $C$ and show that they are exposed. The {\em dimension} of a face is the dimension of the smallest affine subspace that contains this face.
\begin{proposition}\label{prop:FacesOfC}
The following singletons are the only zero-dimensional faces of $C$:
$$
F_{00} = \{p_0\};\quad  F_{0i}(t) = \{\gamma_i(t)\},\quad  t\in (0,T],\; i\in I.
$$
The only one-dimensional faces of $C$ are the following line segments: 
$$
\begin{array}{l}
F_{11}(\theta) = \co \{\gamma_1(\theta),\gamma_3(t_\theta)\},\\
F_{12}(\theta) = \co \{\gamma_4(\theta), \gamma_2(t_\theta)\},
\end{array}
  \qquad \theta \in (0,T];
$$
$$
F_{13} = \co \{p_1,p_2\};\quad F_{14} = \co \{p_3, p_4\};\quad F_{15} = \co \{p_2, p_3\}.
$$
The only two-dimensional faces of $C$ are 
$$
F_{21} = \co \{p_1,p_2, p_3\},\quad  F_{22} = \co \{p_2,p_3, p_4\},\quad F_{23} = \co \{\gamma_1, \gamma_2\}, \quad 
F_{24} = \co \{\gamma_3, \gamma_4\}.
$$
All these faces are exposed.
\end{proposition}
\begin{proof} It is evident from the plot in Fig.~\ref{fig:00} that the only two-dimensional faces of $F$ are, indeed, $F_{21}$, $F_{22}$, $F_{23}$ and $F_{24}$, and that they are exposed. It is also clear that $F_{1i}$, $i=3,4,5$ are one-dimensional exposed faces of $C$, and that the singletons in $\cup_{i=1}^4\gamma_i$ are the only zero-dimensional faces of $C$. It remains to show that the one-dimensional sets $F_{11}(\theta)$ and $F_{12}(\theta)$ are exposed faces of $C$, that all zero-dimensional faces are exposed and that $C$ does not have any other faces.

Fix $\theta \in (0,T]$. By direct substitution we obtain for all $t$
\begin{equation}
\langle \gamma_1(t),y_\theta\rangle   = \cos t_\theta \left(\cos(t-\theta)-\cos \theta \right),  \quad \langle \gamma_3(t),y_\theta\rangle  = \sin \theta \left(\cos(t-t_\theta)-\cos t_\theta \right) \label{eq:020a}
\end{equation}
Observe that for $t,\theta \in [0,T]$ we have $\cos (t-\theta) <1$ for all $t\neq \theta$, and $\cos (t-\theta) = 1$ when $t=\theta$. Hence, relations \eqref{eq:020a} yield
\begin{equation}\label{eq:020}
\langle \gamma_1(\theta), y_\theta\rangle =\cos t_\theta \left(1-\cos \theta \right) = d_\theta, \quad \langle \gamma_3(t_\theta), y_\theta\rangle = \sin \theta \left(1-\cos t_\theta \right)= d_\theta;
\end{equation}
and 
\begin{equation}\label{eq:021}
\begin{array}{ll}
\langle \gamma_1(t), y_\theta\rangle  < \cos t_\theta(1-\cos \theta) = d_\theta & \forall t\in [0,T]\setminus\{ \theta\},\\
\langle \gamma_3(t), y_\theta\rangle  < \sin \theta(1-\cos t_\theta) = d_\theta & \forall t\in [0,T]\setminus\{ t_\theta\}.
\end{array}
\end{equation}
Also, by direct substitution 
\begin{equation}
\langle \gamma_2(t),y_\theta\rangle   = \cos t_\theta \left(\sin \theta - \sin(t+\theta)\right),  \quad \langle \gamma_4(t),y_\theta\rangle   = \sin\theta \left(\sin t_\theta - \sin(t+t_\theta)\right)\label{eq:020b}
\end{equation}
Observe that for $t,\theta \in [0,T]$ we have $\sin (t+\theta)\geq \sin \theta$, hence,
\begin{equation}\label{eq:022}
\langle \gamma_2(t), y_\theta\rangle \leq 0 < d_\theta, \quad \langle \gamma_4(t), y_\theta\rangle \leq 0 < d_\theta   \quad \forall t\in [0,T].
\end{equation}
From \eqref{eq:020},\eqref{eq:021} and \eqref{eq:022} we deduce that for every $\theta \in (0,T]$ the pair $(y_\theta, d_\theta)$ exposes $F_{11}(\theta)$, but no other points of $C$. Hence, for every $\theta \in (0,T]$ the set  $F_{11}(\theta)$ is an exposed face.  

Similarly, it can be shown that the sets $F_{12}(\theta)$ are exposed faces of $C$. In this case  $y_\theta$ must be replaced by the symmetric normals
$$
y'_\theta = (\cos t_\theta \cos \theta, \sin \theta \cos t_\theta, -\sin t_\theta \sin \theta).
$$

We show that the zero-dimensional faces of $C$ are exposed by demonstrating the exposing planes for each of them. Fix $\theta \in (0,T]$ and consider the face $F_{03}(t_\theta)$ (recall that $t_\theta:[0,T]\to [0,T]$ is a bijection by Proposition~\ref{prop:TTheta}). We show that this face is exposed by $(y_3(\theta),d_\theta)\in \R^{3+1}$,
$$
\quad y_3(\theta)  = y_\theta +(0,0,1).
$$
By direct substitution for all $t\in [0,T]$ we have 
\begin{align}\label{eq:035}
\langle \gamma_1(t), y_3(\theta)\rangle & = \langle \gamma_3(t), y_\theta\rangle + \cos t -1 , & 
\langle \gamma_2(t), y_3(\theta)\rangle & = \langle \gamma_2(t), y_\theta\rangle - \sin t ,\notag \\
\langle \gamma_3(t), y_3(\theta)\rangle & = \langle \gamma_3(t), y_\theta\rangle, & 
\langle \gamma_4(t), y_3(\theta)\rangle & = \langle \gamma_4(t), y_\theta\rangle. 
\end{align}
hence, from the relations \eqref{eq:020}-\eqref{eq:022} and \eqref{eq:035} we have
\begin{equation}\label{eq:034}
\langle \gamma_3(t_\theta), y_3(\theta)\rangle = d_\theta, \qquad    \langle \gamma_3(t), y_3(\theta)\rangle < d_\theta, \quad t \in [0,T]\setminus\{t_\theta\}
\end{equation}
\begin{equation}\label{eq:036}
\langle \gamma_1(t), y_3(\theta)\rangle   <d_\theta, \quad 
\langle \gamma_2(t), y_3(\theta) \rangle <d_\theta, \qquad 
\langle \gamma_4(t), y_3(\theta)\rangle <d_\theta \quad \forall t\in (0,T]. 
\end{equation}
It is clear from \eqref{eq:034} and \eqref{eq:036} that for all $\theta\in [0,T]$ the face $F_{03}(t_\theta)$ is exposed. It can be shown analogously that the symmetrically located faces $F_{02}(t_\theta)$ are exposed. In this case the exposing hyperplanes are defined by $(y_2(\theta), d_\theta)$, where 
$$
y_2(\theta) = y'_\theta+(1,0,0).
$$

We show that for every $\theta\in [0,T]$ the face $F_{01}(\theta)$ is exposed by the pair $(y_1(\theta),d_1(\theta))\in \R^{3+1}$, where 
$$
y_1(\theta)  = (1,-\sin \theta, \cos \theta),\quad 
d_1(t) = 1-\cos \theta.
$$
Indeed, for every $\theta\in [0,T]$ we have 
\begin{align*}
\langle \gamma_1(\theta),y_1(\theta)\rangle & = 1-\cos \theta = d_1(\theta),\\
\langle \gamma_1(t),y_1(\theta)\rangle & = \cos(t-\theta)-\cos \theta < 1-\cos \theta = d_1(\theta) \quad \forall t\in [0,T]\setminus \{\theta\},\\
\langle \gamma_2(t),y_1(\theta)\rangle & = \sin \theta -\sin (t+\theta) <0 \leq d_1(\theta)  \quad \forall t\in (0,T],\\
\langle \gamma_3(t),y_1(\theta)\rangle & = -\sin t -\sin \theta(1-\cos t) <0 \leq d_1(\theta) \quad \forall t\in (0,T],\\
\langle \gamma_4(t),y_1(\theta)\rangle & = (\cos t -1)-\sin t \sin \theta <0 \leq d_1(\theta) \quad \forall t\in (0,T],
\end{align*}
hence, $F_{01}(\theta)$ is exposed. Analogously, we can show that for every $\theta \in (0,T]$ the symmetric face $F_{04}(\theta)$ is exposed by
$$
y_4(\theta)  = (\cos \theta, \sin \theta, 1),\quad 
d_4(t) = 1-\cos \theta.
$$


It remains to show that we have listed all faces of $C$.  From Proposition~\ref{prop:TTheta} the mapping $t_\theta: (0,T]\to (0,T]$ is continuous and bijective, hence, every point on the curve $\gamma_1$ is connected with a point on $\gamma_3$ by a one-dimensional face in a continuous manner, and vice versa. Hence, these faces $F_{11}(\theta)$ together with the point $p_0$ cover the part of the surface of $C$ bounded by $\gamma_1$, $\gamma_3$ and $[p_1,p_3]$. Similar argument works for the part of the surface bounded by $\gamma_2$, $\gamma_4$ and $[p_2, p_4]$. It is evident from the plot in Fig.~\ref{fig:01} that the rest of the surface of $C$ is covered by the two-dimensional faces $F_{2i}$, $i =1,\dots, 4$. Hence, if we have missed any faces, they must belong to either $F_{2i}$, $i= 1,\dots,4$, $F_{1i}(\theta)$, $\theta \in (0,T]$ or $F_{00}=\{p_0\}$. It is evident that we have already listed all zero- and one-dimensional faces that comprise the relative boundaries of the aforementioned faces. 
\end{proof}

\begin{proposition}\label{prop:CFExp} The set $C'$ is facially exposed.
\end{proposition}
\begin{proof} Since $C'$ is obtained from $C$ via an affine transformation, and hence facial exposedness of $C$ yields facial exposedness of $C'$, it is sufficient to show that $C$ is facially exposed. Observe that any proper face of $C$ is at most two-dimensional (a face $F \unlhd C$ is proper if $F\neq C$), hence, Propositon~\ref{prop:FacesOfC} implies that all proper faces of $C$ are exposed. Hence, $C$ is exposed.
\end{proof}

\begin{proposition}\label{prop:SpanF} Let 
$F = \cone \left(\{1\}\times (2\co\{\gamma_3, \gamma_4\}+\{c\})\right)$, where $c = (1/2,0,1/2)$. Then $F^\perp =  \lspan \{(1, 0,0,-2)\}$.
\end{proposition}
\begin{proof} 
Observe that the points $q_i = (1,2p_i+c)\in \R^4$, $i= 0,3,4$, where $p_i$'s are defined by \eqref{eq:defp}, belong to $F$. Moreover, it is not difficult to observe that the affine hull $\aff \{p_0,p_3,p_4\}$ coincides with $\aff\{\gamma_3,\gamma_4\}$. Hence, $\lspan F = \lspan\{q_i,\, i\in \{0,3,4\}\}$. For any point $y = (y_0,\bar y)\in F^\perp$ we have  
$$
\langle y, q_i\rangle = 0, \quad i= 0,3,4,
$$
or, equivalently,
$$
\left[
\begin{array}{cccc}
	1 & \tfrac{1}{2} & 0& \tfrac{1}{2} \\
	1 & \tfrac{1}{2}-\sqrt{2}    & 2-\sqrt{2} & \tfrac{1}{2}\\ 
	1 & \sqrt{2}-\tfrac{3}{2}    & \sqrt{2} & \tfrac{1}{2} 
\end{array}
\right]y
= 0.
$$
Solving this linear system, we obtain $y_1=y_2=0$, $y_3=-2y_0$ (where $(y_1, y_2, y_3)=\bar y$), hence, $F^\perp =  \lspan \{(1, 0,0,-2)\}$.
\end{proof}

%

\subsection{Proof of Theorem~\ref{thm:counterex}}\label{s:proof}

The set $K$ defined by \eqref{eq:DefK} is obviously a convex cone and is closed by Proposition~\ref{prop:FacesKfromC}. We first prove that $K$ is facially exposed (Proposition~\ref{prop:KFExp})
and then show that $K$ is not nice (Proposition~\ref{prop:KNotNice}). These two results together yield Theorem~\ref{thm:counterex}.

\begin{proposition}\label{prop:KFExp} The cone $K$ is facially exposed.
\end{proposition}
\begin{proof} 
By Proposition~\ref{prop:CFExp} the set $C'$ is facially exposed, hence, by Proposition~\ref{prop:IfCFExpKFExp} the cone $K= \cone(\{1\}\times C)$ is facially exposed.
\end{proof}

We need the following technical result for the proof of Proposition~\ref{prop:KNotNice}.
\begin{proposition}\label{prop:Technical} For every $\alpha \in \R$ there exists $t_\alpha > 0$ such that 
$$
\varphi_\alpha(t) = \alpha(\cos t -1) + \sin t >0 \quad \forall t\in (0, t_\alpha). 
$$ 
\end{proposition}
\begin{proof} First assume that $\alpha \leq 0$. Then $\varphi_\alpha(t)>0$ for all $t\in (0,\tfrac{\pi}{2})$, and we let $t_\alpha:= \tfrac{\pi}{2}$.
Consider the case when $\alpha >0 $. Observe that for all $t>0$
$$
\cos t > 1-\frac{t^2}{2}, \qquad \sin t > t- \frac{t^3}{6},
$$
therefore, for $t>0$ we have 
$$
\alpha(\cos t -1) + \sin t > \alpha\left(1-\frac{t^2}{2}-1\right)+t- \frac{t^3}{6} = t\left(1-\frac{t}{2}\alpha -\frac{t^2}{6}\right).
$$
Choose $t_\alpha>0$ such that 
$$
\frac{t}{2}\alpha +\frac{t^2}{6}<1 \quad \forall t\in (0,t_\alpha),
$$
hence, $\varphi_\alpha(t) >0$ for all $t\in (0,t_\alpha)$.
\end{proof}

\begin{proposition}\label{prop:KNotNice} The cone $K$ is not nice.
\end{proposition}
\begin{proof}
Observe that there is a one-to-one correspondence between the faces of $C$ and $C'$, namely, for every  $E_0 \unlhd C$ the set $2E_0+\{c\}$ is a face of $C'$ and vice versa. In particular, the face $F_{24} = \co \{\gamma_3, \gamma_4\}$ of $C$ (see Proposition~\ref{prop:FacesOfC}) corresponds to $E := 2 F_{24}+\{c\}$, hence, $E\unlhd C'$. Proposition~\ref{prop:FacesKfromC} yields $F = \cone \left(\{1\}\times E\right) \unlhd K$.
By Proposition~\ref{prop:SpanF}
$$
F^\perp = \lspan\{u\}, \qquad u = (1,0,0,-2).
$$
We show that $q := (-1, 0, -1, 2)\in \overline {(K^\circ+F^\perp)}\setminus (K^\circ+F^\perp)$, hence the set $K^\circ+F^\perp= - (K^*+F^\perp)$ is not closed, and $K $ is not nice.

We first show that $q\notin (K^\circ+F^\perp)$. Assume the contrary. Then $q=p+\lambda u$ for some $p\in K^\circ, \lambda \in \R$. Let
$$
\gamma'(t) := (1, 2\gamma_1(t)+c), \quad t\in \left[0,  T\right].
$$
Since $\gamma'= \gamma'([0,T])\subset (\{1\}\times C')\subset K$, and $p\in K^\circ$, 
\begin{equation}\label{eq:002}
\psi(t): = \langle \gamma'(t), p\rangle \leq 0 \quad \forall t\in [0,  T].
\end{equation}
Observe that
$$
\psi(t) =\langle \gamma'(t), p\rangle= \langle \gamma'(t), q-\lambda u\rangle = 2(2(\lambda+1)(\cos t -1) +\sin t).
$$
Let $\alpha: = 2(\lambda+1)$ in Proposition~\ref{prop:Technical}. Then there exists $t_\alpha>0$
such that $\psi(t)=2\varphi_\alpha(t)>0$ for all $t\in (0,t_\alpha)$. This contradicts \eqref{eq:002}, hence, our assumption is wrong and $q\notin K^\circ+F^\perp$. 

It remains to show that 
\begin{equation}\label{eq:003}
q\in \overline{K^\circ+F^\perp}.
\end{equation}
It follows from \cite[Remark~1]{PatakiFExpNice} that $\overline{K^\circ+F^\perp} =F^\circ$, hence, \eqref{eq:003} is equivalent to showing 
$q\in F^\circ$.
By direct substitution we have for all $t\in [0,T]$  
\begin{equation}\label{eq:030}
\langle q,(1,2\gamma_3(t)+c)\rangle = 2(\cos t-1 ) \leq 0,\qquad 
\langle q,(1,2\gamma_4(t)+c)\rangle  = - 2\sin t \leq 0.
\end{equation}
Recall that $F = \cone (\{1\}\times (2\{\gamma_3\cup \gamma_4\}+c))$, hence, \eqref{eq:030} yields 
$$
\langle q , z\rangle \leq 0 \quad \forall z\in F,
$$
therefore $q\in F^\circ$, and \eqref{eq:003} holds.
\end{proof}

\section*{Acknowledgements} The author is grateful to Minghui Liu and G\'abor Pataki for correcting several mistakes and inaccuracies in an earlier version of the manuscript, and to the anonymous referees for their relevant and insightful remarks, which improved the quality of the paper. 

\bibliographystyle{plain}
\bibliography{conic}

\end{document}